\documentclass[11pt,twoside]{amsart}

\usepackage{latexsym}
\usepackage{amssymb}
 \usepackage{vmargin}
\usepackage{amscd}
\usepackage{stmaryrd}
\usepackage{euscript}
\usepackage{mathrsfs}
\usepackage[all]{xy}
  \usepackage{xr}

\usepackage{xr}
\usepackage
{hyperref}

 \setmargins{32mm}{20mm}{14.6cm}{22cm}{1cm}{1cm}{1cm}{1cm}

\newcommand\la{\leftarrow}

\newcommand\id{\mathrm{id}}

\newcommand\ten{\otimes}

\renewcommand\H{\mathrm{H}}
\newcommand\z{\mathrm{Z}}

\newcommand\N{\mathbb{N}}
\newcommand\Z{\mathbb{Z}}
\newcommand\Q{\mathbb{Q}}

\newcommand\R{\mathbb{R}}
\newcommand\Cx{\mathbb{C}}

\newcommand\bA{\mathbb{A}}

\newcommand\bH{\mathbb{H}}

\newcommand\bL{\mathbb{L}}

\newcommand\bO{\mathbb{O}}

\newcommand\C{\mathcal{C}}

\newcommand\cC{\mathcal{C}}
\newcommand\cD{\mathcal{D}}

\newcommand\cT{\mathcal{T}}

\renewcommand\O{\mathscr{O}}

\newcommand\sA{\mathscr{A}}

\newcommand\sF{\mathscr{F}}

\newcommand\sI{\mathscr{I}}

\newcommand\sO{\mathscr{O}}

\newcommand\fX{\mathfrak{X}}
\newcommand\fY{\mathfrak{Y}}

\renewcommand\L{\Lambda}

\newcommand\m{\mathfrak{m}}

\newcommand\Alg{\mathrm{Alg}}

\newcommand\Mod{\mathrm{Mod}}

\newcommand\Hom{\mathrm{Hom}}
\newcommand\Map{\mathrm{Map}}
\newcommand\map{\mathrm{map}}

\newcommand\loc{\mathrm{loc}}

\newcommand\Aff{\mathrm{Aff}}

\newcommand\Com{\mathrm{Com}}

\newcommand\<{\langle}
\renewcommand\>{\rangle}
\newcommand\Lim{\varprojlim}
\newcommand\LLim{\varinjlim}

\newcommand\ho{\mathrm{ho}\!}
\newcommand\into{\hookrightarrow}
\newcommand\onto{\twoheadrightarrow}

\newcommand\xra{\xrightarrow}
\newcommand\xla{\xleftarrow}

\newcommand\alg{\mathrm{alg}}

\newcommand\bt{\bullet}
\newcommand\by{\times}

\newcommand\et{\acute{\mathrm{e}}\mathrm{t}}

\newcommand\an{\mathrm{an}}

\newcommand\ind{\mathrm{ind}}
\newcommand\pro{\mathrm{pro}}

\newcommand\pd{\partial}

\newcommand\Coh{\cC\mathit{oh}}

\newcommand\DCrit{\mathrm{DCrit}}

\renewcommand\alg{\mathrm{alg}}

\newcommand\co{\colon\thinspace}

\newcommand\oC{\mathbf{C}}
\newcommand\oE{\mathbf{E}}

\newcommand\oR{\mathbf{R}}
\newcommand\oP{\mathbf{P}}
\newcommand\oL{\mathbf{L}}

\newcommand\oT{\mathbf{T}}

\newcommand\uleft\underleftarrow
\newcommand\uline\underline
\newcommand\uright\underrightarrow

\newtheorem{theorem}{Theorem}[section]
\newtheorem{proposition}[theorem]{Proposition}

\newtheorem{lemma}[theorem]{Lemma}
\newtheorem*{theorem*}{Theorem}
\newtheorem*{proposition*}{Proposition}
\newtheorem*{corollary*}{Corollary}
\newtheorem*{lemma*}{Lemma}
\newtheorem*{conjecture*}{Conjecture}

\theoremstyle{definition}
\newtheorem{definition}[theorem]{Definition}
\newtheorem*{definition*}{Definition}

\newtheorem*{notation*}{Notation}

\theoremstyle{remark}
\newtheorem{example}[theorem]{Example}
\newtheorem{examples}[theorem]{Examples}
\newtheorem{remark}[theorem]{Remark}
\newtheorem{remarks}[theorem]{Remarks}

\newtheorem{assumption}[theorem]{Assumption}

\newtheorem*{example*}{Example}
\newtheorem*{examples*}{Examples}
\newtheorem*{remark*}{Remark}
\newtheorem*{remarks*}{Remarks}
\newtheorem*{exercise*}{Exercise}
\newtheorem*{property*}{Property}
\newtheorem*{properties*}{Properties}

\begin{document}

\begin{abstract}
We give a formulation for derived analytic geometry built from commutative differential graded algebras equipped with entire functional calculus on their degree $0$ part, a theory well-suited to developing shifted  Poisson structures and quantisations.  In the complex setting, we show that this formulation recovers equivalent derived analytic spaces and stacks to those coming from  Lurie's structured topoi. In non-Archimedean settings, there is a similar comparison, but  for derived dagger analytic spaces and stacks, based on overconvergent functions. 
\end{abstract}

\title{A differential graded model for derived analytic geometry} 

\author{J. P. Pridham}

\maketitle

\section*{Introduction}

In this paper, we develop  a formulation for derived analytic geometry based on differential graded objects, by applying the approach of Carchedi and Roytenberg from \cite{CarchediRoytenbergHomological}. In this case, the objects are commutative differential graded (dg) algebras equipped with entire functional calculus (EFC) on their degree $0$ part. These are examples of dg algebras for a Fermat theory, so share many formal similarities with the dg $\C^{\infty}$ rings used in derived differential geometry.
All commutative Banach (and indeed Fr\'echet) algebras
naturally carry EFC structures,
 with operations given not just by polynomials but by all entire holomorphic functions. 

Entire functional calculus should not be confused with holomorphic functional calculus, the basis of Dubuc and Taubin's analytic rings \cite{DubucTaubin},  Lurie's approach to derived complex analytic geometry \cite{lurieDAG5,lurieDAG9} and Porta and Yu's derived non-Archimedean analytic geometry \cite{PortaYuNonArch}. The difference is that EFC does not include partially defined operations (corresponding to open subspaces of affine space) as part of the structure. Because EFC is the 
Lawvere theory giving the 
closest algebraic approximation to commutative Banach algebras, 
 our approach is almost closer in spirit to that of \cite{BambozziBenBassatKremnizer}. 

Our main results show that entire functional calculus is sufficient to study geometric objects in a derived analytic setting. Proposition \ref{locmodelprop} is the key to many technical constructions, giving a Quillen equivalent model structure on a category of dg algebras equipped with entire functional calculus in which all \'etale maps (i.e. local biholomorphisms) are cofibrations. As in Proposition \ref{DHomprop2}, this gives rise to a functor from our setup to Lurie's derived analytic spaces, even though the latter carry much more structure \emph{a priori}. In a similar vein, Proposition \ref{DHomprop} gives a full and faithful contravariant $\infty$-functor  from Porta's  derived Stein spaces to complex dg algebras with EFC, so our simpler formulation leads to the same objects.

In proving these results, we cannot simply treat the underlying geometry as a black box as in Lurie's approach \cite{lurieDAG5} to derived analytic geometry, but instead need to exploit specific features. The most important ingredient is given by the various classical embedding theorems for Stein spaces \cite{forster,Remmert,wiegmann,kiehlThmAB,luetkebohmert}, allowing us to express analytic spaces locally as equalisers of finite-dimensional affine spaces --- these neighbourhoods correspond to compact objects in the category of rings with entire functional calculus. 
Another key ingredient is given by Proposition \ref{dgshrink} and Assumption \ref{openetale}, isolating the features of open immersions and \'etale maps which allow us to analytically localise or Henselise one of our dg algebras without affecting its homotopy class. These are used in \S \ref{cotsn} to establish that cotangent complexes for Fermat theories  have the properties we desire, notably that those   associated to  \'etale maps of Stein spaces are trivial (Lemma \ref{Cunramlemma}).

Because our setup is formulated for any rational Fermat theory, these results also include  versions  in non-Archimedean analytic geometry. Since we need embedding theorems, we have to be able to reduce to Stein or pro-Stein spaces. This means that we cannot work with arbitrary derived rigid analytic spaces, but we can work with partially proper spaces 
or with derived dagger analytic spaces. 
In other words, the lack of flexibility for our setup compared with \cite{lurieDAG9,PortaYuNonArch} forces us to work with overconvergent functions as in \cite{GrosseKloenne}, but that might be no bad thing.

\tableofcontents

\section{Rings with entire functional calculus}


\subsection{Preliminary properties of  Stein algebras}

Given a complex analytic space $X$, we write $\sO_X$ for the sheaf of  holomorphic functions on $X$, and $\sO(X)$ for the ring $\Gamma(X,\sO_X)$ of global holomorphic functions.

\begin{lemma}\label{flatlemma}
 Given a local biholomorphism $f \co X \to Y$ of Stein spaces, the induced ring homomorphism $f^{\sharp} \co \sO(Y) \to \sO(X)$ is flat.
\end{lemma}
\begin{proof}
 The category of $\sO(Y)$-modules is equivalent to the ind-category of  finitely presented $\sO(Y)$-modules.  It thus suffices to show that the functor $M \mapsto M\ten_{\sO(Y)}\sO(X)$  on the category of finitely presented $\sO(Y)$-modules preserves monomorphisms. 

Serre's Three Lemma implies that the category of  coherent $\sO_Y$-modules is abelian, while
 Cartan's Theorem B \cite[Fundamental Theorem B, \S IV.4, p.124]{GrauertRemmertStein} implies that the functor $\Gamma(Y,-)$ from coherent $\sO_Y$-modules to $\sO(Y)$-modules is exact.  The functor $-\ten_{\sO(Y)}\sO_Y$ is right exact, and sends $\sO(Y)$ to $\sO_Y$, so it follows from exactness that $\Gamma(Y,M\ten_{\sO(Y)}\sO_Y) \cong M$ for any finitely presented $\sO(Y)$-module $M$.
Thus the functor $-\ten_{\sO(Y)}\sO_Y$  must also preserve monomorphisms of finitely presented $\sO(Y)$-modules, since it has  a  quasi-inverse $\Gamma(Y,-)$ on its essential image.

Since $f$ is a local biholomorphism, the functor $f^* \co \Coh(\sO_Y) \to \Coh(\sO_X)$ on coherent sheaves is exact, so we 
 can  factorise $-\ten_{\sO(Y)}\sO(X)$ as the composition $FP\Mod(\sO(Y)) \xra{-\ten_{\sO(Y)}\sO_Y } \Coh(\sO_Y) \xra{f^*} \Coh(\sO_X) \xra{\Gamma(X,-)} \Mod(\sO(X))$ of functors preserving monomorphisms, where $\Mod$ denotes the category of modules and $FP\Mod$ the category of finitely presented modules.
\end{proof}

\begin{definition}\label{cancellative}
 As in \cite{anel}, we say that a class $\oP$ of morphisms in a category satisfies right  cancellation if, for diagrams $X \xra{f} Y \xra{g} Z$,   we have that $g$ lies in $\oP$ whenever $f$ and $g \circ f$ do so. Similarly, a class $\oP$ satisfies left  cancellation if $f$ lies in $\oP$ whenever $g$ and $g \circ f$ do so.
\end{definition}

The following is immediate:
\begin{lemma}\label{Steincancellemma}
 In the category of Stein spaces, the classes of open immersions and of local biholomorphisms are both left cancellative.
\end{lemma}

\subsection{Complex algebras with entire functional calculus}

We now consider a natural complex-analytic analogue of the $\C^{\infty}$ rings of \cite{dubuc,MoerdijkReyes}. Although these will not model all analytic spaces,  they will suffice for our building blocks after restricting to objects locally of finite presentation.

\begin{definition}\label{EFCdef}
Recall (cf. \cite{pirkovskiiHFG}) that a $\Cx$-algebra $A$ with entire functional calculus (or EFC $\Cx$-algebra for short)  is given by a product-preserving set-valued  functor $\Cx^n \mapsto A^n$  on the category with objects $\{\Cx^n\}_{n \ge 0}$ and morphisms consisting of complex-analytic maps. 

Explicitly, the set $A$ is equipped, for every complex-analytic function $f \in \sO(\Cx^n)$, with an operation $\Phi_f \co A^n \to A$. These operations are required to be compatible in the sense that given functions $g_i \in \sO(\Cx^{m_i})$, we must have 
\[
 \Phi_{f \circ (g_1, \ldots, g_n)}= \Phi_f \circ (\Phi_{g_1}, \ldots, \Phi_{g_n}) \co A^{\sum_{i=1}^n m_i}\to A.  
\]
\end{definition}
Thus an EFC $\Cx$-algebra is a $\Cx$-algebra $A$ with a systematic and consistent way of evaluating expressions of the form $\sum_{m_1, \ldots, m_n=0}^{\infty} \lambda_{m_1, \ldots, m_n} a_1^{m_1}\cdots a_n^{m_n}$ in $A$ whenever the coefficients $\lambda_{m_1, \ldots, m_n} \in \Cx$ satisfy $\lim_{\sum m_i \to \infty } |\lambda_{m_1, \ldots, m_n}|^{1/\sum m_i}= 0$.
In particular, every EFC $\Cx$-algebra is a commutative $\Cx$-algebra, with  addition and multiplication coming from the functions $(x,y) \mapsto x+y$ and
  $(x,y) \mapsto xy$ in $\sO(\Cx^2)$.

\begin{examples}\label{EFCexamples}
 For every complex-analytic space $X$, the set $\sO(X)$ of global holomorphic functions on $X$ naturally has the structure of an EFC $\Cx$-algebra. as does any quotient of $\sO(X)$ by an ideal (see Examples \ref{Fermatexamples} below). 
In particular, Artinian rings with residue field $\Cx$ (for instance  the dual numbers $\Cx[x]/x^2$) can naturally be regarded as EFC $\Cx$-algebras. 

Any filtered colimit of EFC $\Cx$-algebras is naturally an EFC $\Cx$-algebra. This includes examples such as the ring $\Cx^{\N}/c_{00}(\Cx)$ of infinite sequences modulo finite sequences, which does not arise as the ring of holomorphic functions on any analytic space; we might think of it  as the ring functions on the inverse system of cofinite subsets of $\N$, even though the limit of that system is empty.   
\end{examples}

From a categorical perspective, the forgetful functor from EFC $\Cx$-algebras to sets has a left adjoint, 
 which sends a set $S$ to the EFC $\Cx$-algebra
\[
\sO(\Cx^S):=  \LLim_{\substack{T \subset S}\\ \text{finite}} \sO(\Cx^T),
\]
and  EFC $\Cx$-algebras are algebras for the resulting monad structure on the functor $S \mapsto \sO(\Cx^S)$.

The category of EFC $\Cx$-algebras has all small limits and colimits. In particular, there is a coproduct, which we denote by $\odot$, with the property that $\sO(M \by N) \cong \sO(M)\odot\sO(N)$. We also denote pushouts by $A\odot_BC$. Because all quotient rings of an  EFC $\Cx$-algebra are  EFC $\Cx$-algebras as in Examples \ref{EFCexamples}, it follows that if we have surjective EFC $\Cx$-algebra homomorphisms $\sO(\Cx^S) \onto A$, $\sO(\Cx^T) \onto C$, then the pushout of $A \la B \to C$ can be expressed as
\[
 A\odot_BC:= \sO(\Cx^{S \sqcup T})\ten_{(\sO(\Cx^S)\ten\sO(\Cx^T))}(A\ten_BC).
\]


\begin{definition}
We say that an EFC $\Cx$-algebra is finitely presented if it arises as a quotient of $\sO(\Cx^n)$ by a finitely generated ideal, for some finite $n$.
\end{definition}

Since the monad $S \mapsto \sO(\Cx^{S})$ preserves filtered colimits, we immediately have: 
\begin{lemma}\label{indFP}
 The category of EFC $\Cx$-algebras is equivalent to the category of ind-objects of the category of finitely presented EFC $\Cx$-algebras.
\end{lemma}

\begin{definition}
 We say that a coherent sheaf $\sF$ on a Stein space $\sO_X$ is globally finitely generated if it is  generated as a sheaf of  $\O_X$-modules by a finite set of global sections. We say that $\sF$ is globally finitely presented if it is globally finitely generated and the kernel of the induced surjection $\sO_X^{\oplus m} \onto \sF$ is also globally finitely generated. 
\end{definition}

\begin{definition} 
 We say that a Stein space $X$ is finitely embeddable if if it admits a closed embedding  $i \co X \to \Cx^n$ for some finite $n$. We say that a Stein space is 
globally finitely presented if it is finitely embeddable  in such a way that $i_*\sO_X$ is globally finitely presented (equivalently, such that  the defining ideal is  globally finitely generated).
\end{definition}

\begin{remark}\label{embeddingrmk}
 By the Embedding Theorem of  \cite{wiegmann}, a Stein space $X$ is finitely embeddable  
 whenever it is finite dimensional and the tangent dimensions $\dim_{\Cx} \m_x/\m_x^2$ are globally bounded, for $x \in X$ and  $\m_x$ the maximal ideal of $\O_{X,x}$. In particular, by Remmert's Theorem \cite{Remmert}, this applies to all Stein manifolds.

By \cite[Theorem 4.3]{forster}, a coherent sheaf $\sF$ on a finite-dimensional Stein space $Y$  is globally finitely generated if the dimensions $\dim_{\Cx} \sF_y/\m_y\sF_y$ are globally bounded for $y \in Y$. Thus a Stein space $X \subset \Cx^n$ with defining ideal $\sI$ is globally finitely presented if $\sup_{y \in \Cx^n}  \dim_{\Cx} (\sI_y/\m_y\sI_y) < \infty$, and it suffices to take the supremum over $y \in X$ because $\sI_y \cong \O_{\Cx^n,y}$ for $y \notin X$. 

In particular, all Stein manifolds $X$  are  globally finitely presented, since for any closed embedding $X \subset \Cx^n$, we will have $\dim_{\Cx} (\sI_x/\m_x\sI_x) =n -\dim X$ when $x \in X$. In general, a finitely embeddable Stein space need not be globally finitely presented, an example being the closed subspace of $\Cx^2$ defined by the ideal $\bigcap_{n \ge 1}(x-n,y)^n$. However, any Stein space admits an open cover by globally finitely generated Stein spaces.  An example of a Stein space which is not finitely embeddable is given by $\coprod_{n \ge 0} \Cx^n$. 
\end{remark}


  
\begin{lemma}\label{Homlemma}
 Given an analytic space $Y$ and a finitely embeddable Stein space $X$,  the set $\mathrm{Hol}(Y,X)$ of holomorphic maps from $Y$ to $X$ is isomorphic to the set $\Hom_{EFC}(\sO(X),\sO(Y))$ of EFC $\Cx$-algebra homomorphisms from $\sO(X)$ to $\sO(Y)$.
\end{lemma}
\begin{proof}
Since $X$ is finitely embeddable, there exists a closed embedding $X \subset \Cx^n$, defined by a coherent ideal $\sI$. This gives an injection $\mathrm{Hol}(Y,X)\into \mathrm{Hol}(Y,\Cx^n)=\sO(Y)^n$, which in turn is isomorphic to the set of EFC $\Cx$-algebra homomorphisms from $\sO(\Cx^n)$ to $\sO(Y)$. The subset $\mathrm{Hol}(Y,X)\into \mathrm{Hol}(Y,\Cx^n)$ consists of morphisms which annihilate $\sI$. If we write $I:=\Gamma(\Cx^n,\sI)$, this
 these are precisely the same as EFC $\Cx$-algebra homomorphisms from $\sO(X)= \sO(\Cx^n)/I$ to $\sO(Y)$, since $\sI \cong I\ten_{\sO(\Cx^n)}\sO_{\Cx^n}$ by  Cartan's Theorem A \cite[Fundamental Theorem A, \S IV.4, p.124]{GrauertRemmertStein}.
\end{proof}

The following does not have an analogue for infinite coproducts of affine schemes. It suggests that the finitely embeddable hypothesis in  Lemma \ref{Homlemma} might be relaxed.  
\begin{lemma}\label{Homlemmacoprod}
Given an analytic space $Y$ and finitely embeddable Stein spaces $X_i$ for $i \in \N$, the natural map
\[
 \Hom(Y, \coprod X_i) \to \Hom_{EFC}(\prod_i \sO(X_i), \sO(Y)),
\]
coming from the expression $\sO(\coprod X_i) =\prod_i \sO(X_i)$, is an isomorphism.
\end{lemma}
\begin{proof}
Let $a =(1,2,3,\ldots) \in  \prod_i \sO(X_i)$, so $\Phi_{(\exp(2\pi i z)-1)}(a)=0$. Then, for any EFC $\Cx$-algebra homomorphism $\theta \co \prod_i \sO(X_i)\to \sO(Y)$, we must have $\Phi_{(\exp(2\pi i z)-1)}(\theta a)=0$, so the holomorphic function $\theta(a) \co Y \to \Cx$ takes values in the discrete Stein space $\Z$ (the vanishing locus of $\exp(2\pi i z)-1$).
This gives a decomposition $Y=\coprod_i Y_i$ with $Y_i= \theta(a)^{-1}(i)$, so $\sO(Y_i)= \sO(Y)/(\theta(a)-i)$. Thus $\theta$ is the product over $i \in \N$ of   EFC $\Cx$-algebra homomorphisms  $\sO(X_i) \to \sO(Y_i)$, so as in Lemma \ref{Homlemma}, it  comes from a unique
Stein space homomorphism
\[
 Y=\coprod_i Y_i \to \coprod X_i.
\]
\end{proof}

\begin{proposition}\label{fpalgprop}
 The functor $X \mapsto \sO(X)=\Gamma(X,\sO_X)$ gives  a contravariant equivalence of categories between globally finitely presented Stein spaces and finitely presented EFC $\Cx$-algebras. 
\end{proposition}
\begin{proof}
If $X$ is a globally finitely presented Stein space, take an embedding $X \subset \Cx^n$ with globally finitely generated ideal $\sI$ (necessarily coherent). Cartan's Theorem B \cite[Fundamental Theorem B, \S IV.4, p.124]{GrauertRemmertStein} thus implies that $\H^1(\Cx^n, \sI)=0$, so for $I:=\Gamma(\Cx^n,\sI)$ we have
$
 \sO(X)\cong \sO(\Cx^n)/I
$
by the long exact sequence of cohomology. Since $\sI$ is globally finitely generated, the $\sO(\Cx^n)$-module $I:=\Gamma(\Cx^n,\sI)$ is finitely generated, so $\sO(X)$ is a finitely presented EFC $\Cx$-algebra.

It follows from Lemma \ref{Homlemma} that the functor is full and faithful. To see that it is essentially surjective, take an EFC $\Cx$-algebra $A$ of the form $\sO(\Cx^n)/I$, with $I$ finitely generated, so there exists a surjection $\alpha \co \sO(\Cx^n)^{\oplus m} \onto I$ for some finite $m$. Then the ideal sheaf $\sI:= I\sO_{\Cx^n} \subset \sO_{\Cx^n}$ is necessarily coherent and globally finitely generated, so defines a globally finitely presented Stein space $i \co X \into \Cx^n$. 

Since $A$ is finitely presented as an $\sO(\Cx^n)$-module, it follows from the proof of Lemma \ref{flatlemma} that 
$A \cong \Gamma(\Cx^n, A\ten_{\sO(\Cx^n)}\sO_{\Cx^n})$. But now
\[
 \Gamma(\Cx^n, A\ten_{\sO(\Cx^n)}\sO_{\Cx^n})= \Gamma(\Cx^n, \sO_{\Cx^n}/\sI)=\Gamma(\Cx^n,i_*\sO_X) = \Gamma(X,\O_X),
\]
completing the proof.
\end{proof}

\begin{remarks} 
An analogue of Proposition \ref{fpalgprop} in the $\C^{\infty}$-setting is given by  
\cite[Proposition 10]{dubuc}, which shows that every finitely generated ideal of $\C^{\infty}(\R^n)$ is of local character (i.e. germ-determined), and hence that the $\C^{\infty}$ spectrum functor is full and faithful when restricted to finitely presented $\C^{\infty}$-rings via \cite[Theorem 13]{dubuc}.

Meanwhile, \cite{pirkovskiiHFG} shows that finitely embeddable Stein spaces correspond to holomorphically finitely generated Fr\'echet algebras. These are harder to characterise algebraically than the globally finitely presented Stein spaces, the difficulty being a description of closed ideals $I \le \sO(\Cx^n)$. A necessary condition is that 
\[
I \cong \sO(\Cx^n)\by_{\prod_{x\in \Cx^n} \sO_{\Cx^n,x}}\prod_{x\in \Cx^n} I\sO_{\Cx^n,x},
\]
 but this does not obviously guarantee that the ideal sheaf $I \sO(\Cx^n)$ is coherent.
 \end{remarks}

\begin{lemma}\label{fpmodlemma}
Given a Stein space $X$, the global sections functor gives a contravariant equivalence of categories between globally finitely presented coherent sheaves on $X$ and finitely presented $\sO(X)$-modules. 
 \end{lemma}
\begin{proof}
 By \cite[Theorem 2.1]{forster}, the global sections functor $\Gamma(X,-)$ from coherent sheaves on $(X,\sO_X)$ to $\Gamma(X,\O_X)$-modules is full and faithful; it is also exact, by Cartan's Theorem B. The functor automatically sends  globally finitely presented coherent sheaves to finitely presented $\sO(X)$-modules. An inverse functor is given by sending a finitely presented $\sO(X)$-module $M$ to globally finitely presented the sheaf $M\ten_{\sO(X)}\sO_X$; that this is an inverse follows by exactness, since $\sO(X)= \Gamma(X, \sO_X)$.
\end{proof}

\begin{definition}
 Say that a morphism $ f \co A \to B$ of EFC $\Cx$-algebras is a finite localisation if it is the pullback of some morphism $ f \co \bar{A} \to \bar{B}$ of finitely presented EFC $\Cx$-algebras corresponding via Proposition \ref{fpalgprop} to an open immersion of Stein spaces.
\end{definition}

\subsection{Non-Archimedean EFC algebras}

The description after Definition \ref{EFCdef} can be adapted to $K$-algebras for any complete normed field $K$. 
For this definition, an EFC $\R$-algebra $A$ would correspond to an EFC-algebra $A\ten_{\R}\Cx$, equipped with a semilinear complex conjugation operation $a \mapsto \bar{a}$ satisfying $\overline{\Phi_f(a_1, \ldots, a_n)}= \Phi_{\bar{f}}(\bar{a}_1, \ldots, \bar{a}_n) $. We now consider a non-Archimedean analogue, fixing  a complete normed non-Archimedean field $K$.

\begin{definition}\label{EFCnonArchdef}
Define a  $K$-algebra $A$ with entire functional calculus (or EFC $K$-algebra for short)  to be a product-preserving set-valued  functor $\bA_K^n \mapsto A^n$  on the full subcategory of rigid analytic varieties with objects the affine spaces $\{\bA_K^n\}_{n \ge 0}$. 
\end{definition}
Thus  an EFC $K$-algebra $A$  is a $K$-algebra equipped with a systematic and consistent way of evaluating expressions of the form $\sum_{m_1, \ldots, m_n=0}^{\infty} \lambda_{m_1, \ldots, m_n} a_1^{m_1}\cdots a_n^{m_n}$ in $A$ whenever the coefficients $\lambda_{m_1, \ldots, m_n} \in K$ satisfy $\lim_{\sum m_i \to \infty } |\lambda_{m_1, \ldots, m_n}|^{1/\sum m_i}= 0$. 

\begin{examples}\label{EFCexamplesnonArch}
 For every rigid analytic space $X$ over $K$, the set $\sO(X):=\Gamma(X,\sO_X)$ of global analytic functions on $X$ naturally has the structure of an EFC $K$-algebra, as does any quotient of $\sO(X)$ by an ideal (see Examples \ref{Fermatexamples} below). 
Any filtered colimit of EFC $K$-algebras is also naturally an EFC $K$-algebra. 
\end{examples}

\begin{definition}
We say that an EFC $K$-algebra is finitely presented if it arises as a quotient of $\sO(\bA_K^n)$ by a finitely generated ideal, for some finite $n$.
\end{definition}
As in the Archimedean setting, the  category of EFC $K$-algebras is equivalent to the category of ind-objects of the category of finitely presented EFC $K$-algebras.

\begin{definition}
 We say that a Stein space over $K$ is finitely embeddable if if it admits a closed embedding  in $\bA_K^n$ for some finite $n$.
We say that a Stein space over $K$ is 
globally finitely presented if it 
admits a closed embedding  in $\bA_K^n$ for some finite $n$
  in such a way that the defining ideal is  generated as a sheaf of  $\sO_{\bA^n}$-modules by a finite set of global sections.
\end{definition}


\begin{remark}\label{embeddingrmk2}
By \cite[Theorem 4.23]{luetkebohmert}, every  finite-dimensional Stein space over $K$ with a global bound on the local embedding dimensions   is finitely embeddable. Under the equivalence \cite[Theorem  2.27]{GrosseKloenne} between partially proper rigid spaces and partially proper dagger spaces, the corresponding statement is true for $K$-dagger Stein spaces.

The observations of Remark \ref{embeddingrmk} now all adapt, replacing \cite[Theorem 4.3]{forster} with 
the following argument. Given a coherent sheaf $\sF$ on a finitely embeddable Stein space $Y$, we may apply  \cite[Theorem 4.23]{luetkebohmert} to the Stein space $(Y, \sO_Y\oplus \sF)$ to see that $\sF$  is globally finitely generated if the dimensions $\dim_K \sF_y/\m_y\sF_y$ are globally bounded for $y \in Y$.
 In particular, this implies that any non-singular Stein space is globally finitely generated, and that every Stein space admits an open cover by globally finitely presented Stein spaces.
\end{remark}

\begin{lemma}\label{HomlemmaNonarch}
 Given a rigid analytic space $Y$ and a finitely embeddable Stein space $X$ over $K$, 
  the set $\Hom(Y,X)$ of  morphisms from $Y$ to $X$ is isomorphic to the set $\Hom_{EFC}(\sO(X),\sO(Y))$ of EFC $K$-algebra homomorphisms from $\sO(X)$ to $\sO(Y)$.

Given a $K$-dagger space $Y$, with $X$  a finitely embeddable $K$-dagger Stein space or a $K$-dagger affinoid, 
  the set $\Hom(Y,X)$ of  morphisms from $Y$ to $X$ is isomorphic to the set $\Hom_{EFC}(\sO(X),\sO(Y))$ of EFC $K$-algebra homomorphisms from $\sO(X)$ to $\sO(Y)$.
\end{lemma}
\begin{proof}
 The proof of Lemma \ref{Homlemma} now adapts to give the  statements for Stein spaces, substituting \cite[Theorem 2.4]{kiehlThmAB} for Cartan's Theorems. 

By definition, a $K$-dagger affinoid $X$ is the ringed space associated to some quotient $W_n/I$ of the Washnitzer algebra of overconvergent functions on a polydisc. Since the Washnitzer algebra is Noetherian and is given by a nested union $\bigcup_{\rho>1}T_n(\rho)$ of Tate algebras, we can choose generators $f_1, \ldots,f_m$ for the ideal $I$ lying in some $T_n(r)$. Then for any $\rho\le r$, we have a Stein space $X_{\rho}$ given by the vanishing locus of $f_1, \ldots,f_m$  on the open polydisc of radius $\rho$. Thus the ringed space $X$ arises as the inverse limit $X=\Lim_{r \ge\rho>1} X_{\rho}$ of Stein spaces. Since the forgetful functor from EFC-algebras to sets preserves filtered colimits, we also have $\sO(X)=\LLim_i\sO(X_i)$ in the category of $\oT$-algebras, completing the proof.   
\end{proof}

In particular,  this gives a full and faithful contravariant functor from finitely embeddable Stein spaces over $K$ to  EFC $K$-algebras.

\begin{lemma}\label{HomlemmacoprodNonarch}
Given a rigid analytic space $Y$ and, for $i \in \N$, finitely embeddable Stein spaces $X_i$ over $K$,
  the natural map
\[
 \Hom(Y, \coprod X_i) \to \Hom_{EFC}(\prod_i \sO(X_i), \sO(Y)),
\]
coming from the expression $\sO(\coprod X_i) =\prod_i \sO(X_i)$, is an isomorphism.
\end{lemma}
\begin{proof}
The proof of Lemma \ref{Homlemmacoprod} adapts, using the sequence
 $a =(1,\pi^{-1},\pi^{-2},\ldots) \in  \prod_i \sO(X_i)$ for an element $\pi \in K$ with $|\pi|<1$. This satisfies $\Phi_{f}(a)=0$ for $f(z)= \prod_{r\ge 0} (1-\pi^rz)$, which converges on the whole of the affine line because the $n$th coefficient has norm $|\pi|^{n(n-1)/2}$, so decays sub-exponentially.  
\end{proof}

The following  follow with the same proofs as Proposition \ref{fpalgprop} and Lemma \ref{fpmodlemma}, using Kiehl's theorems  \cite{kiehlThmAB}  in place of  Forster's and Cartan's theorems.

\begin{proposition}\label{fpalgprop2}
 The functor $X \mapsto \sO(X)=\Gamma(X,\sO_X)$ gives  a contravariant equivalence of categories between globally finitely presented Stein spaces over $K$ and finitely presented EFC $K$-algebras. 
\end{proposition}

\begin{lemma}\label{fpmodlemma2}
Given a Stein space $X$, the global sections functor gives a contravariant equivalence of categories between globally finitely presented coherent sheaves on $X$ and finitely presented $\sO(X)$-modules. 
 \end{lemma}

\begin{definition}\label{finlocKdef}
 Say that a morphism $ f \co A \to B$ of EFC $K$-algebras is a finite localisation if it is the pullback of some morphism $ f \co \bar{A} \to \bar{B}$ of finitely presented EFC $K$-algebras corresponding via Proposition \ref{fpalgprop} to an open immersion of Stein spaces.
\end{definition}

\section{Generalisations of derived rings} 
We now recall some results from \cite{CarchediRoytenberg,CarchediRoytenbergHomological} on differential graded algebras with respect to Fermat theories.

\subsection{Fermat theories}

\begin{definition}
 As for instance in \cite{AdamekRosickyVitale}, a Lawvere theory is a small category $\oT$ closed under finite products, and generated under products by an object $\bA^1_{\oT} \in \oT$. Write $\bA^n_{\oT} \in \oT$ for the $n$-fold product of $\bA^n$.
\end{definition}

\begin{definition}
 Given a Lawvere theory $\oT$, an $\oT$-algebra is a product preserving functor $F$ from $\oT$ to sets.
\end{definition}

Thus an $\oT$-algebra consists of a set $A:=F(\bA^1_{\oT})$ equipped with operations $\Phi_f \co A^n \to A$ for each $f \co \bA^n_{\oT} \to \bA^1_{\oT}$ satisfying various compatibility conditions. Sending $A$ to its underlying set gives a forgetful functor from $\oT$-algebras to sets, and this has a left adjoint which we denote by $S \mapsto \oT[S]$; for finite sets, the underlying set is given by 
\[
 \oT[x_1, \ldots,x_n]\cong \Hom_{\oT}(\bA^n_{\oT},\bA^1_{\oT}).
\]
This gives another way of describing $\oT$-algebras, as algebras for the monad sending a set $S$ to the set underlying $\oT[S]$. This monad commutes with filtered colimits, so in particular, we have $\oT[S]= \LLim_{\substack{T \subset S\\ \text{finite}}} \oT[T]$.

\begin{definition}
Given a commutative ring $R$, the Lawvere theory $\Com_R$ is the full subcategory of affine schemes on the affine spaces $\bA^n_R$ over $R$. We simply write $\Com:=\Com_{\Z}$. 
\end{definition}

A $\Com_R$-algebra is just a commutative $R$-algebra, since $\Hom_{\Com_R}(\bA^n_{\Com_R}, \bA^1_{\Com_R})\cong R[x_1, \ldots,x_n]$, and then we set $\Phi_f(a_1, \ldots,a_n):=f(a_1, \ldots,a_n)$.

\begin{definition}
 As in \cite{DubucKock}, a Fermat theory is a Lawvere theory $\oT$ equipped with a morphism $\Com \to \oT$ such that for all
$f \in \oT[x, z_1 , \ldots , z_n]$ there exists a unique $g \in \oT[x, y, z_1 , \ldots, z_n ]$ such that
\[
f (x, z_1 , \ldots , z_n) - f (y, z_1 , \ldots , z_n) = (x - y)g(x, y,z_1 , \ldots , z_n ).
\]
A rational Fermat theory is a Fermat theory equipped with an extension $\Com_{\Q} \to \oT$ of the morphism above.
\end{definition}

Thus an algebra over a Fermat theory is a commutative ring with well-behaved extra structure, and over a rational Fermat theory, the underlying ring is a $\Q$-algebra.

\begin{examples}\label{Fermatexamples}
The main example we will consider of a rational Fermat theory is the theory $\bH$ of \cite[Example 2.16]{CarchediRoytenberg}, which we will refer to as EFC. This consists of the complex-analytic manifolds $\Cx^n$ with holomorphic functions between them. Algebras for this theory are EFC $\Cx$-algebras. In this case, we have $\oT[x_1, \ldots,x_n]=\sO(\Cx^n)$, the EFC $\Cx$-algebra of entire holomorphic functions on $\Cx^n$.

We  also consider the case when $\oT$ consists of the rigid analytic varieties $\bA_K^n$ over a complete non-Archimedean field $K$, algebras  for this theory being EFC $K$-algebras. In this case, for $\pi \in K$ with $|\pi|<1$, we have $\oT[x_1, \ldots,x_n]= \Lim_r K\< \pi^r x_1, \ldots,\pi^r x_n\> $, the EFC $K$-algebra of analytic functions on $\bA_K^n$; these can be characterised as power series with coefficients satisfying $\lim_{\sum m_i \to \infty } |\lambda_{m_1, \ldots, m_n}|^{1/\sum m_i}= 0$.

Another example is given by the real manifolds $\R^n$ and smooth functions; these give rise to the theory of $\C^{\infty}$-rings, with $\oT[x_1, \ldots,x_n]=\C^{\infty}(R^n)$, the $\C^{\infty}$-ring of infinitely differentiable functions on $\R^n$.

Note that as in \cite{DubucKock} (\cite[Corollary 2.8]{CarchediRoytenberg}), any quotient of an $\oT$-algebra by an ideal is again an $\oT$-algebra. Any filtered colimit of $\oT$-algebras is also naturally an $\oT$-algebra, essentially because the functor $\oT[-]$ preserves filtered colimits.
\end{examples}

\begin{definition}
Given a rational Fermat theory $\oT$, denote the left adjoint to the forgetful functor from  $\oT$-algebras to commutative $\Q$-algebras by $A \mapsto A_{\oT}$. Explicitly, if $A= \Q[S]/I$, then $A_{\oT} \cong \oT[S]/I \oT[S]$.
\end{definition}

\begin{definition}
We say that an $\oT$-algebra  is finitely generated if it arises as a quotient of $\oT[x_1, \ldots, x_n]$  for some finite $n$.
We say that an $\oT$-algebra  is finitely presented if it arises as a quotient of $\oT[x_1, \ldots, x_n]$ by a finitely generated ideal, for some finite $n$.
\end{definition}

Note that finitely presented $\oT$-algebras $A$ are the compact objects in the sense that $\Hom_{\Alg(\oT)}(A,-)$ preserves filtered colimits. The category of $\oT$-algebras is then immediately equivalent to the ind-category of finitely presented $\oT$-algebras (write an arbitrary $\oT$-algebra as $\oT[S]/I$, then note that it can be expressed as the filtered colimit of $\oT$-algebras $\oT[T]/J$ with $T \subset S$ finite and $J\subset I \cap \oT[T]$ finitely generated).   

\begin{definition}
 Say that a morphism of $\oT$-algebras is flat if the underlying morphism of commutative rings is so.
\end{definition}

For a Fermat theory $\oT$, the category of $\oT$-algebra contains all small limits and colimits. As for EFC $\Cx$-algebras, we denote the pushout of $A \la B \to C$ by 
$A\odot_BC$. Because all quotient rings of an  $\oT$-algebra are  $\oT$-algebras, it follows that if we choose a sets $S$ and $T$ of generators for $A$ and $C$, then
\[
 A\odot_BC\cong \oT[S \sqcup T]\ten_{(\oT[S]\ten\oT[T])}(A\ten_BC).
\]

\subsubsection{Modules}

Following \cite{Q}, there is a natural notion of modules over any object $A$ in a category with finite limits, given by taking abelian group objects in the category of $A$-augmented objects. For  any Fermat theory $\oT$, these correspond to modules $M$ over the commutative ring underlying $A$, by forming the $\oT$-ring $A \oplus M$ as follows.

\begin{definition}\label{oplusdef}
 Given an $\oT$-ring $A$ and an $A$-module $M$, we define an $\oT$-ring structure on $A \oplus M$ by setting $\Phi_f^{A\oplus M}(a_1+ m_, \ldots a_n+m_n):= \Phi_f^A(a_1, \ldots, a_n) + \sum_{i=1}^n \Phi_{\frac{\pd f}{\pd x_i}}(a_1, \ldots, a_n)m_i$ for all $f \in \oT[x_1, \ldots,x_n]$. We then say that a map $\delta \co A \to M$ is an $\oT$-derivation if induces an $\oT$-ring homomorphism  $A \to A \oplus M$  sending $a$ to $(a, \delta a)$.
\end{definition}
Explicitly, adapting \cite[Definition 5.10]{joyceAGCinfty},  the condition for  $\delta$ to be a derivation amounts to saying that for all $f \in \oT[x_1, \ldots,x_n]$, we have
\[
 \delta \Phi_f(a_1, \ldots, a_n) = \sum_{i=1}^n \Phi_{\frac{\pd f}{\pd x_i}}(a_1, \ldots, a_n) \delta a_i.
\]

\subsection{EFC-Differential graded algebras}

From now on, we fix a rational Fermat theory $\oT$.

The following correspond to the differential graded $\oT$-algebras of \cite[Definition 4.14]{CarchediRoytenbergHomological} following  \cite[Example 2.16]{CarchediRoytenberg}. 

\begin{definition}\label{EFCDGAdef}
Define an $\oT$-differential graded algebra ($\oT$-DGA for short) to be a chain complex $A_{\bt}=(A_*, \delta)$ of $\Q$-vector spaces equipped with:
\begin{itemize}
 \item an associative graded multiplication, graded-commutative in the sense that $ab= (-1)^{\bar{a}\bar{b}}ba$ for all $a,b \in A$, where $\bar{a}$ is the parity of $a$ (i.e. the degree modulo $2$), and 
\item an enhancement of the $\Q$-algebra structure on $\z_0A=\ker(\delta \co A_0 \to A_{-1})$ to an $\oT$-algebra structure,
\end{itemize}
such that $\delta$ is a graded derivation in the sense that $\delta(ab)= \delta(a)b + (-1)^{\bar{a}}a\delta(b)$ for all $a,b \in A$.
\end{definition}

\begin{examples}
Every $\oT$-algebra can be regarded as an $\oT$-differential graded algebra concentrated in degree $0$, so Definition \ref{EFCDGAdef} includes all the examples of Examples \ref{EFCexamples} when $\oT$ is EFC.

It also includes functions on shifted cotangent bundles $T^*M[n]$ of complex  manifolds $M$, with $\sO(T^*M[n])$ given by the free graded-commutative algebra over $\sO(M)$ generated by analytic sections $T_M$ of the tangent bundle placed in chain degree $n$, and with trivial differential $\delta$.

A more interesting example is given by the derived critical locus $\DCrit(M,f)$ of a function $f \in \sO(M)$. The EFC-differential graded algebra  $\sO( \DCrit(M,f) )$ is given by the chain complex
\[
 \sO(M) \xla{ \neg df} T_M \xla{ \neg df} \L^2_{\sO(M)}T_M \xla{ \neg df} \ldots,
\]
so $\H_0\sO( \DCrit(M,f) )$ consists of functions on the critical locus of $f$, and we have $ \DCrit(M,0)= T^*M[1]$.
\end{examples}

\begin{definition}
 We say that a morphism $A_{\bt} \to B_{\bt}$ of $\oT$-DGAs is a quasi-isomorphism if it induces an isomorphism $\H_*(A_{\bt}) \to \H_*(B_{\bt})$ on homology groups.
\end{definition}

\begin{proposition}\label{stdmodelprop}
There is a cofibrantly generated model structure (which we call the standard model structure) on the category of non-negatively graded $\oT$-DGAs  in which  a morphism $A_{\bt} \to B_{\bt}$ is
\begin{enumerate}
 \item  a weak equivalence if it is a quasi-isomorphism;
\item a fibration if it is surjective in strictly positive chain degrees.
\end{enumerate}

There is also a model structure  for simplicial EFC $\Cx$-algebras in which weak equivalences are $\pi_*$-isomorphisms and fibrations are Kan fibrations, and Dold--Kan normalisation combines with the Eilenberg--Zilber shuffle product \cite[Definitions 8.3.6 and 8.5.4]{W} to give a right Quillen equivalence $N$ from simplicial EFC $\Cx$-algebras to our standard model structure. 
\end{proposition}
\begin{proof}
As observed in the introduction to \cite{CarchediRoytenbergHomological}, the standard model structure is induced by the inclusion-truncation adjunction between non-negatively graded $\oT$-DGAs and all $\oT$-DGAs applied to the model structure of  \cite[Theorem 6.10]{CarchediRoytenbergHomological}.

The model structure on simplicial $\oT$-rings is a special case of  \cite[Remarks II.4.2] {QHA}, and the proof of \cite[Corollary 2.2.10]{nuitenThesis} (phrased for $\C^{\infty}$-rings, but applicable to any Fermat theory) then shows that $N$ is a right Quillen equivalence. 
\end{proof}

\section{Generalisations of localisation and henselisation}

We now generalise the ideas from \cite[Lemma \ref{stacks2-dgshrink}]{stacks2}, finding ways to replace the standard model structure with Quillen equivalent model structures having many more cofibrant objects, thus allowing us to calculate  mapping spaces and cotangent complexes more easily in cases of interest.

Recall that we are fixing a rational Fermat theory $\oT$. We now consider classes of morphisms behaving like open immersions or \'etale maps.
\begin{assumption}\label{openetale}
From now on, we assume that $\oE$ is a class  
  of flat morphisms between finitely presented $\oT$-algebras such that
\begin{enumerate}
\item $\oE$ is closed under pushouts along arbitrary maps of finitely presented $\oT$-algebras;
 \item $\oE$ satisfies right cancellation (Definition \ref{cancellative});
\item \label{factorcdn}  if we have a map $f \co A \to B$ in $\oE$, and a map  $i \co B \to C$ with  surjective composition $i \circ f \co A \onto C$, then there exists a factorisation $B \xra{j} B' \to C$ of $i$, with $j$ a morphism in $\oE$, such that the induced map $B'\odot_AC \to C$ is an isomorphism.
\end{enumerate}
\end{assumption}

If we think of $\oT$-algebras as rings of functions on spaces, then the final condition says that if we have $X \to Y$ in $\oE$  and a map $i \co Z \to X$ with $Z \to Y$ closed, then there exists $X' \to X$ in $\oE$ with the projection  $X'\by_YZ \to Z$ being an isomorphism, and the projection $X'\by_YZ\to X' \to X$ then being $i$.

\begin{definition}
 Refer to a morphism of $\oT$-algebras as an $\oE$-morphism if it arises as a pushout of a morphism in $\oE$ between finitely presented $\oT$-algebras.
\end{definition}

\begin{lemma}\label{factorlemma}
 If the other conditions of Assumption \ref{openetale} hold, then Assumption \ref{openetale}.\ref{factorcdn} also holds in either of the following situations:
\begin{enumerate}
 \item all morphisms of $\oE$ are epimorphisms, or
\item $\oE$ contains a map $\oT[x]\to D$ with $D/(x^2-x) \cong \oT[x]/x$.
\end{enumerate}
\end{lemma}
In other words, the second situation says that there should be an object \'etale over the affine line which contains $0$ but not $1$.
\begin{proof}
 In the first situation, the map  $C\to  B\odot_AC $ is an epimorphism equipped with a retraction, so is an isomorphism.

In the second situation, take $f \co A \to B$ in $\oE$, and  $i \co B \to C$ with  $i \circ f\co A \to C$ surjective. Surjectivity implies that $C\odot_AB \cong C\ten_{A}B \cong B/(f\ker(i \circ f))$. The pushout  $C \to C\ten_{A}B$ of $f$
 is automatically an $\oE$-morphism, and by right cancellation the map $ C\ten_{A}B \to C$ induced by $i$ is also an $\oE$-morphism. 
  In particular, it corresponds to a flat closed immersion of affine schemes, so the defining ideal is generated by an idempotent $e \in C\ten_{A}B$. Lift $e$ to an element (not necessarily idempotent) $\tilde{e} \in B$. 

By pushing out the morphism $\Q[x]_{\oT} \to D$ from the hypothesis along $x \mapsto \tilde{e}$, we obtain 
 an $\oE$-morphism  $B \to B':=B\odot_{\Q[\tilde{e}]_{\oT}} D$.  It suffices to show that $B'\odot_AC \cong C$. Observe that 
\begin{align*}
C\ten_{A}B'&\cong (C\ten_{A}B)\odot_{\Q[\tilde{e}]_{\oT}} D\\
& \cong (C\ten_{A}B)\odot_{(\Q[e]/(e^2-e))_{\oT}}(\Q[e]/(e^2-e))_{\oT}\odot_{\Q[e]_{\oT}} D\\
& \cong (C\ten_{A}B)\odot_{(\Q[e]/(e^2-e)_{\oT})}(D/(e^2-e)) \\
&\cong (C\ten_{A}B)\odot_{(\Q[e]/(e^2-e)_{\oT})} (\Q[e]/(e))\\ 
&\cong (C\ten_{A}B)/(e)\\
&\cong C,
\end{align*}
as required, the penultimate step following because pushouts along surjections of $\oT$-algebras are just quotients by ideals. 
\end{proof}

\begin{examples}\label{openetaleexamples}
The simplest example satisfying Assumption \ref{openetale} is given by taking $\oT=\Com_R$ and $\oE$ the class of morphisms corresponding to open immersions between finitely presented affine schemes over $R$. Another example on the same category is given by the class of \'etale morphisms. The latter is the largest possible class, since for any morphism in $\oE$, right cancellation implies that the relative diagonal must also be in $\oE$, hence flat. But a flat closed immersion is a local isomorphism, so morphisms in $\oE$ are finitely presented, flat, and unramified, hence \'etale. 

In general, any morphism $A \to B$ of finitely presented $\oT$-algebras must be unramified in the sense that the diagonal $B\odot_AB \to B$ is also flat. If the class of flat unramified morphisms  is closed under pushouts, then it automatically satisfies right cancellation.

The main examples which will concern us are when $\oT$ is the theory of EFC $\Cx$-algebras and we take $\oE$ to consist either of finite localisations or of \'etale maps (i.e. local biholomorphisms) of finitely presented EFC $\Cx$-algebras. These are flat by Lemma \ref{flatlemma}, are automatically closed under pushouts, and satisfy right cancellation by Lemma \ref{Steincancellemma}. For Assumption \ref{openetale}.\ref{factorcdn}, we use Lemma \ref{factorlemma}, taking the finite localisation $\sO(\Cx) \to \sO(\Cx\setminus \{1\})$, since $\{0,1\}\by_{\Cx}(\Cx  \setminus \{1\})= \{0\}$. 
 
We will also consider the case when $\oT$ is the theory of EFC $K$-algebras over a non-Archimedean field $K$, and take $\oE$ to correspond to open immersions or \'etale morphisms (in the sense of \cite[\S 8.1]{FresnelvdPut}) of Stein spaces; between finitely presented EFC-algebras, \'etale morphisms $A \to B$ are just those inducing isomorphisms $\Omega^1_A\ten_AB \to \Omega^1_B$ on the modules of EFC differentials.
Then Assumption \ref{openetale} is satisfied, applying Lemma \ref{factorlemma} to the immersion of the open unit disc in $\bA^1_K$ to satisfy  Assumption \ref{openetale}.\ref{factorcdn}.

Another example satisfying Assumption \ref{openetale}  is given by taking $\oT$ to be $\C^{\infty}$ and $\oE$ to be the class of morphisms corresponding to open immersions of finitely presented affine $\C^{\infty}$-schemes.
\end{examples}

\subsection{$\oE$-localisation}

We now establish a partial analogue of  \cite[Lemma \ref{stacks2-dgshrink}]{stacks2}, replacing each $\oT$-DGA with a form of localisation or henselisation to give a Quillen equivalent model structure. For EFC-DGAs, this results in a model structure for which EFC $\Cx$-algebras of holomorphic functions on Stein submanifolds of $\Cx^n$ are cofibrant.

\begin{definition}\label{locdef}
 Given a morphism $ f \co A \to B$ of $\oT$-algebras, define the $\oE$-localisation of $(A/B)^{\loc_{\oE}}$ of $A$ along $B$ to be the $\oT$-algebra given by the colimit of all EFC $\Cx$-algebras $C$ equipped with maps $A \xra{g} C \to B$ factorising $f$, with $g$ an $\oE$- morphism.
\end{definition}

\begin{remark}
 Since   $\oE$ is  closed under pushout and composition, right cancellation ensures that the colimit in Definition \ref{locdef} is taken over a filtered category.
\end{remark}

\begin{lemma}\label{UFSlemma}
The factorisation $A \to  (A/B)^{\loc_{\oE}} \to B$ of Definition \ref{locdef} is, up to unique isomorphism, the only factorisation $A \xra{g} C \xra{u} B$ of  $ f \co A \to B$ for which $g$ is a filtered colimit of morphisms in $\oE$ and $u$ has the unique right lifting property with respect to morphisms in $\oE$.
\end{lemma}
\begin{proof}
 Assumption \ref{openetale} (less the final condition) ensures that the conditions of \cite[Lemma 2, Lemma 11 and Theorem 14]{anel} are satisfied, giving  a unique factorisation system on the category of $\oT$-algebras, in which the left class is $\ind(\oE)$.
\end{proof}

\begin{examples}
Taking $\oE$ to be the class of finite localisations of commutative rings gives rise to the Zariski factorisation system of \cite[\S 4.2]{anel}, in which case  $(A/B)^{\loc_{\oE}}$ is localisation of $A$ along $B$, and $(A/B)^{\loc_{\oE}} \to B$ is conservative in the sense that an element of $(A/B)^{\loc_{\oE}}$ is a unit if and only if its image in $B$ is so.

Taking $\oE$ to be the class of \'etale morphisms of commutative rings gives rise to the \'etale factorisation system of \cite[\S 4.2]{anel}, in which case $(A/B)^{\loc_{\oE}}$ is the relative Henselisation of $A$ along $B$, with $A \to (A/B)^{\loc_{\oE}}$ ind-\'etale and $(A/B)^{\loc_{\oE}} \to B$ Henselian.

There are similar interpretations when we take $\oT$ to be EFC, with $\oE$ respectively the finite localisations or local biholomorphisms of finitely presented EFC $\Cx$-algebras. In the latter case, the morphism $(A/B)^{\loc_{\oE}} \to B$ is determined by having the unique right lifting property with respect to  $\Cx \to \Cx^{\Z}$ and  $\sO(\Delta) \to \sO(\Cx)$.

\end{examples}

\begin{lemma}\label{surjloclemma}
 Given a surjection $A \to C$ of $\oT$-algebras, we have a natural isomorphism 
\[
 C\ten_{A}(A/C)^{\loc_{\oE}} \cong C. 
\]
\end{lemma}
\begin{proof}
 Since $\oE$-localisation and isomorphisms are stable under passage to filtered colimits, it suffices to prove this when $A$ is finitely presented.

Let $\cD$ be the filtered category of finitely presented $\oT$-algebras $B$ equipped with an $\oE$-morphism $f \co A \to B$ and a partial retraction $r \co B \to C$, so $(A/C)^{\loc_{\oE}}= \LLim_{B \in \cD}B$. 
It then suffices to show that the full subcategory $\cD' \subset \cD$ of objects $B$ with $C\ten_{A}B \cong C$ is cofinal. For any object  $A \xra{f} B\xra{r} C$ of $\cD$, Assumption \ref{openetale}.\ref{factorcdn} gives  an object $ A \xra{f} B' \xra{r} C$ of $\cD'$ under $A \xra{f} B \xra{r} C$, as required.
\end{proof}

\begin{definition}
Given a non-negatively graded $\oT$-DGA $A_{\bt}$, we define  $A_{\bt}^{l\oE}$ to be the $\oE$-localisation of $A_{\bt}$    along $\H_0A$, i.e.  the $\oT$-DGA
\[
 A_{\bt}^{l\oE}:= A_{\bt}\ten_{A_0}(A_0/\H_0A)^{\loc_{\oE}}.
\]
\end{definition}

The following is the technical key to all of our constructions which follow. 
\begin{proposition}\label{dgshrink}
For any non-negatively graded $\oT$-DGA $A_{\bt}$, the natural map
\[
 A_{\bt} \to A_{\bt}^{l\oE}
\]
is a quasi-isomorphism.
\end{proposition}
\begin{proof}
 Since the morphism $A_0 \to A_0^{l\oE}$ is in ind-$\oE$, it is flat (Assumption \ref{openetale}), so by flat base change, we have
\[
 \H_n(A_{\bt}^{l\oE})\cong (\H_nA)\ten_{A_0} (A^{l\oE})_0,
\]
and we can rewrite this as $(\H_nA)\ten_{\H_0A}(\H_0A\ten_{A_0} (A^{l\oE})_0)$. Then Lemma \ref{surjloclemma} gives
\[
 \H_0A\ten_{A_0}(A_0/\H_0A)^{\loc_{\oE}} \cong \H_0A, 
\]
which completes the proof.
\end{proof}

\begin{proposition}\label{locmodelprop}
 There is a cofibrantly generated model structure (which we call the $\oE$-model structure) on the category of those non-negatively graded $\oT$-DGAs $A_{\bt}$ with $A_{\bt}\cong A_{\bt}^{l\oE}$, in which  weak equivalence are quasi-isomorphisms and
 fibrations are surjective in strictly positive chain degrees.  The inclusion  functor to all  non-negatively graded $\oT$-DGAs is then a right Quillen equivalence.

Every transfinite composition of $\oE$-morphisms  is a cofibration in this model structure, and in particular any $\oT$-algebra $B$ equipped with an $\ind(\oE)$-morphism $\oT[x_1, \ldots, x_n] \to B$ is cofibrant.
\end{proposition}
\begin{proof}
Since  Proposition \ref{dgshrink} implies that $\H_0(A^{l\oE}_{\bt}) \cong \H_0(A_{\bt})$, it follows from unique factorisation that the functor $A_{\bt} \mapsto A^{l\oE}_{\bt}$ is idempotent, and in particular a monad, with algebras for the monad being   $\oT$-DGAs $A_{\bt}$ with $A_{\bt}\cong A_{\bt}^{l\oE}$.  For the $\oE$-model structure,  factorisations are just given by applying $(-)^{l\oE}$ to the  factorisation in the standard model category of non-negatively graded $\oT$-DGAs. It is straightforward to check  directly that this gives a model structure using Proposition \ref{dgshrink}, or we can appeal to Kan's transfer theorem \cite[Theorem 11.3.2]{Hirschhorn}. Proposition \ref{dgshrink} also ensures that the resulting Quillen adjunction is a Quillen equivalence.

Given a morphism  $f \co A \to B$ in $\oE$, observe that for any  non-negatively graded $\oT$-DGA of the form $C_{\bt}\cong C_{\bt}^{l\oE}$, we have
\[
 \Hom_{\oT-DGA}(B,C_{\bt}) \cong \Hom_{\oT-DGA}(A,C_{\bt})\by_{ \Hom_{\oT}(A,\H_0C_{\bt})}\Hom_{\oT}(B,\H_0C_{\bt}),
\]
since $\Hom_{\oT-DGA}(B,C_{\bt})=\Hom_{\Alg(\oT)}(B,C_0)$, and $C_0 \to \H_0C_{\bt}$ has the unique right lifting property with respect to $\oE$-morphisms, by Lemma \ref{UFSlemma}.
Since any trivial fibration $C_{\bt}\to D_{\bt}$ is an isomorphism on $\H_0$, it follows that $f$ has the (unique) left lifting property with respect to trivial fibrations, so is a cofibration; any pushout or transfinite composition of cofibrations is also automatically a cofibration.

Finally, recall  that $\oT[x_1, \ldots, x_n] $ is cofibrant in the standard model structure, and hence also in the $\oE$-model structure, so for any morphism $\oT[x_1, \ldots, x_n] \to B$ in $\oE$, it follows that $B$ is also cofibrant. 
\end{proof}

\begin{examples}
If we take $\oT$-algebras to be commutative rings, and $\oE$ to be the class of finite localisations (resp. \'etale maps), then $(A/B)^{\loc_{\oE}}$ is localisation (resp. Henselisation) of $A$ along $B$,   Proposition \ref{locmodelprop} gives a model structure on dgas $A_{\bt}$ with $A_0 \to \H_0A$ conservative (resp. Henselian), as followed from \cite[Lemma \ref{stacks2-dgshrink}]{stacks2}. The Henselian model structure has the property that smooth affine algebras are cofibrant. 

If we instead take $\oT$-algebras to be EFC $\Cx$-algebras, and $\oE$ to be finite localisations (resp. \'etale maps), then we obtain a model structure in which rings of holomorphic functions of open Stein submanifolds of $\Cx^n$ (resp. Stein manifolds admitting a local biholomorphism to $\Cx^n$) are cofibrant. 
\end{examples}

\begin{lemma}\label{mapSteinlemma}
 Given an  $\oE$-morphism $\oT[x_1, \ldots, x_n] \to A$ of $\oT$-algebras,  mapping spaces in the $\infty$-category of non-negatively graded $\oT$-DGAs localised at weak equivalences satisfy
\[
 \map(A,B) \simeq \map(\oT[x_1, \ldots, x_n],B)\by_{\Hom_{\Alg(\oT)}( \oT[x_1, \ldots, x_n],\H_0B)}\Hom_{ \Alg(\oT)}(A,\H_0B).
\]
In particular, $\pi_0\map(A,B) \cong \Hom_{\Alg(\oT)}(A,\H_0B)$ and $\pi_i\map(A,B) \cong (\H_iB)^n$ for $i>0$.
\end{lemma}
\begin{proof}
 By Proposition \ref{dgshrink}, we may replace $B$ with its $\oE$-localisation $B^{l\oE}$. Since Proposition \ref{locmodelprop} shows $A$ is cofibrant in the $\oE$-model structure, as in \cite[\S 5.4]{hovey} the mapping space can be calculated by taking a fibrant  simplicial resolution  $\tilde{B}({\bt})$ of $B^{l\oE}$ in the 
$\oE$-model structure, and considering the simplicial set
\[
 i \mapsto \Hom_{\oT-DGA}(A,\tilde{B}(i)). 
\]

Now, since $ \oT[x_1, \ldots, x_n] \to A$ is in $\oE$,
  the unique factorisation property from Lemma \ref{UFSlemma} 
 gives an isomorphism
\begin{align*}
  &\Hom_{\oT-DGA}(A,\tilde{B}(i)) \cong \\
&\Hom_{\oT-DGA}( \oT[x_1, \ldots, x_n],\tilde{B}(i))\by_{\Hom_{\Alg(\oT)}(\oT[x_1, \ldots, x_n],\H_0B)}\Hom_{\Alg(\oT)}(A,\H_0B), 
\end{align*}
so $\map(A,B) \simeq \map(\oT[x_1, \ldots, x_n],B)\by_{\Hom_{\Alg(\oT)}(\oT[x_1, \ldots, x_n],\H_0B)}\Hom_{EFC}(A,\H_0B)$.

The final statement follows because $\map(\oT[x_1, \ldots, x_n],B) \cong (\H_0B)^n$, via the forget-free Quillen adjunction between EFC-DGAs and chain complexes.  
\end{proof}

\subsection{Cotangent complexes}\label{cotsn}

\begin{definition}\label{dgoplusdef}
 Given an $\oT$-DGA $A$ and an $A$-module $M$ in chain complexes, define the $\oT$-DGA $A \oplus M$ by setting the multiplication to be $(a_1,m_1)\cdot (a_2,m_2):= (a_1\cdot a_2, a_1\cdot  m_2 + m_1\cdot a_2)$, with $\oT$-structure on $\z_0(A \oplus M)= \z_0A \oplus \z_0M$ given by Definition \ref{oplusdef}.    
\end{definition}

\begin{definition}\label{Omegadef}
 Given an $\oT$-DGA $A$, define the complex $\Omega^1_A$ to be the $A$-module in chain complexes representing  the functor $M \mapsto \Hom_{\oT-DGA}(A, A \oplus M)\by_{\Hom_{\oT-DGA}(A,A)}\{\id\}$ of closed $\oT$-derivations from $A$ to $M$ of degree $0$. Given a morphism $R \to A$ of $\oT$-DGAs, define $\Omega^1_{A/R}$ to be the cokernel of $\Omega^1_R\ten_RA \to \Omega^1_A$.
\end{definition}

Note that for morphisms $A \la R \to B$ of $\oT$-DGAs, by universality we automatically have an isomorphism 
\[
\Omega^1_{(A\odot_RB)/R}\cong (\Omega^1_{A/R}\ten_A(A\odot_RB)) \oplus ((A\odot_RB)\ten_B \Omega^1_{B/R}).
\]
Meanwhile, properties of idempotents ensure that $\Omega^1_{(A \by B)/R} \cong \Omega^1_{A/R}\by \Omega^1_{B/R}$. 

\begin{lemma}\label{Cunramlemma}
 For any class $\oE$ of morphisms $R \to A$ of $\oT$-algebras satisfying Assumptions \ref{openetale}, the cotangent module $\Omega^1_{A/R}$ is $0$. 
\end{lemma}
\begin{proof}
By the cancellation property, the relative diagonal $\Delta \co A\odot_RA \to A$ lies in $\oE$, so 
 is flat. Since $\Delta$ is surjective, we may argue as in the proof of Lemma \ref{factorlemma}
to give an idempotent $e \in A\odot_RA$  generating $\ker\Delta$, and hence an isomorphism  $A\odot_RA \cong A \by C$ of $\oT$-algebras, for  $C:= (A\odot_RA)/(1-e)$, with $\Delta$ being projection onto $A$. 

We now have $\Delta^*(\Omega^1_{(A\odot_RA)/R}) \cong \Delta^*(\Omega^1_{A/R}\by \Omega^1_{C/R})=  \Omega^1_{A/R}$, but we also have $\Delta^*(\Omega^1_{(A\odot_RA)/R}) \cong \Omega^1_{A/R} \oplus \Omega^1_{A/R}$. Thus the summation map $\Omega^1_{A/R} \oplus \Omega^1_{A/R} \to \Omega^1_{A/R}$ is an isomorphism, so $\Omega^1_{A/R}\cong 0$.
\end{proof}

\begin{lemma}\label{cotlemma}
There is a right Quillen  functor  from the category of pairs  $(A,M)$ of   $\oT$-DGAs and modules to the category of $\oT$-DGAs, given by
\[
(A,M) \mapsto A \oplus M, 
\]
with left adjoint $A \mapsto (A, \Omega^1_A)$. This moreover defines a Quillen adjunction for the $\oE$-model structure of Proposition \ref{locmodelprop}. 
\end{lemma}
\begin{proof}
 For the standard model structure, this is immediate. For the $\oE$-model structure, we need to know that for any non-negatively graded DGA $A$ with $A= A^{l\oE}$, and for any $A$-module $M$ in non-negatively graded chain complexes, we have $(A\oplus M)^{l\oE}\cong A\oplus M$. This amounts to showing that the morphism $A_0 \oplus M_0 \to \H_0A \oplus \H_0M$ has the right lifting property with respect to $\oE$-morphisms. By pullback, we know that the morphism $A_0 \oplus \H_0M \to \H_0A \oplus \H_0M$ has this property, so it suffices to show that    $A_0 \oplus M_0 \to A_0 \oplus \H_0M$ has the property. By adjunction, this follows from Lemma \ref{Cunramlemma}.
\end{proof}

\begin{definition}
 Denote the left-derived functor of $A \mapsto (A, \Omega^1_A)$ by $A \mapsto (A, \oL\Omega^1_A)$. We refer to $\oL\Omega^1_A$ as the cotangent complex of $A$. Given a morphism $R \to A$ of $\oT$-DGAs, write $\oL\Omega^1_{A/R}$ for the cone of the natural map $ \oL\Omega^1_R\ten_RA \to  \oL\Omega^1_A$. 
\end{definition}

\begin{lemma}\label{cotalglemma}
Given a morphism $R \to A$ of non-negatively graded  $\oT$-DGAs with $\H_0R \onto \H_0A$, the relative cotangent complex $\oL\Omega^1_{A/R}$ is quasi-isomorphic to the algebraic cotangent complex $\oL\Omega^{1,\alg}_{A/R}$ of the underlying CDGAs.
\end{lemma}
\begin{proof}
 Since the definition of $\oT$-derivations is purely algebraic in non-zero degrees, we automatically have $\Omega^{1}_{A/A_0} \simeq\Omega^{1,\alg}_{A/A_0}$. Moreover, for a cofibrant $\oT$-DGA $A$, the morphism $A_0 \to A$ is a cofibration of CDGAs, so in general we have $\oL\Omega^{1}_{A/A_0} \simeq\oL\Omega^{1,\alg}_{A/A_0}$  We then just observe that since $\H_0R \onto \H_0A$, we can take compatible cofibrant replacements $\tilde{R} \to R$ and $\tilde{A} \to A$ with $\tilde{R}_0 \cong \tilde{A}_0$, and the result follows. 
\end{proof}

The following is now an immediate consequence of cofibrancy in the $\oE$-model structure of Proposition \ref{locmodelprop}:
\begin{lemma}\label{cotcoflemma}
 Given an $\oE$-morphism $\oT[x_1, \ldots, x_n] \to A$ of $\oT$-algebras,
 the cotangent complex $\oL\Omega^1_{A}$ is quasi-isomorphic to $\Omega^1_A$.
\end{lemma}

\begin{example}\label{cotSteinlemma}
  Given a Stein submanifold $X \subset \Cx^n$, Lemma \ref{cotcoflemma} implies that the cotangent complex $\oL\Omega^1_{\sO(X)}$ of the EFC-algebra $\sO(X)$ of holomorphic functions on $X$ is  modelled by $ \Omega^1_{\sO(X)}$. This is isomorphic to the module $\Gamma(X, \Omega^1_X)$ of global sections  of the sheaf of holomorphic differentials on $X$, because it  follows from Proposition \ref{fpalgprop} and Lemma \ref{fpmodlemma}, that they represent the same functor on Stein modules. 

In fact, the same description holds when $X$ is any Stein manifold, since we may then choose local biholomorphisms $X \xla{\pi} \tilde{X} \to \Cx^n$, with $\pi$ surjective, and  apply Lemma \ref{Cunramlemma} to give $ \oL\Omega^1_{\sO(X)}\ten^{\oL}_{\sO(X)}\sO(\tilde{X}) \simeq \Omega^1_{\sO(X)}\ten_{\sO(X)}\sO(\tilde{X})$, meaning the natural map $\oL\Omega^1_{\sO(X)} \to \Omega^1_{\sO(X)} $ must be a quasi-isomorphism.

Using Proposition \ref{fpalgprop2} and Lemma \ref{fpmodlemma2}, similar results hold for smooth non-Archimedean Stein subspaces $X \subset \bA_K^n$, and indeed for all smooth non-Archimedean Stein spaces.
\end{example}

\subsection{Comparison with structured topoi}

In Lurie's formulation \cite{lurieDAG5,lurieDAG9} of derived analytic geometry, further developed by Porta and Yu \cite{PortaGAGA1,PortaYuNonArch,PortaYuRep}, structure sheaves are required to carry far more data than we have in our setup. This extra data takes the form of spaces of $U$-valued functions for all Stein spaces (or affinoids in the rigid analytic setting). For the complex analytic and overconvergent non-Archimedean settings, we now show that for derived analytic spaces, and even for derived Artin analytic stacks, this extra data can be deduced just from the functions to the affine line, together with their entire functional calculus.

We now let $K$ be either $\Cx$ or a non-Archimedean field. 
\begin{definition}\label{ROdef}
 Given a derived $K$-analytic space $X=(\fX,\bO_X)$ in the sense of \cite[Definition 12.3]{lurieDAG9} or \cite[Definition 2.5]{PortaYuNonArch}, define the non-negatively graded EFC-DGA $\oR\sO(X)$ as follows. The full and faithful embedding of analytic affine spaces in the $\infty$-category of derived stacks gives a product-preserving $\infty$-functor $\bA^n_K \mapsto \map(X,\bA^n_K)$, where $\map(X,Y)$ is the mapping space, regarded as a simplicial set. By \cite[5.5.9.3]{luriehighertopoi}, this can be represented by a simplicial-set valued  functor which preserves products on the nose; in other words, we may regard $\map(X,\bA^1_K)$ as a simplicial EFC-algebra, and we then set 
\[
 \oR\sO(X):= N\map(X,\bA^1_K)
\]
to be the simplicial Dold--Kan normalisation with Eilenberg--Zilber shuffle product as in Proposition \ref{stdmodelprop}.

We also define a hypersheaf $\oR\sO_X$ of non-negatively graded EFC-DGAs on the topos $\fX$ by $(\oR\sO_X)(V):= \oR\sO(V, \bO_X|_V)$.
\end{definition}

\begin{proposition}\label{compcotprop}
 Given a derived $K$-analytic spaces $X$ in the sense of \cite[Definition 12.3]{lurieDAG9} or \cite[Definition 2.5]{PortaYuNonArch} such that the underived truncation $t_0X$ is equivalent to   a  finitely embeddable Stein space, the cotangent complex of \cite[\S 5.2]{PortaYuRep} is given by
\[
 \bL_X^{\an} \simeq \oL\Omega^1_{\oR\sO(X)}\ten^{\oL}_{\oR\sO(X)}\sO_X^{\alg}.
\]
\end{proposition}
\begin{proof}
Applying the functor $\oR\sO$ to an analytic split square-zero extension as in \cite[Definition 5.14]{PortaYuRep}  gives rise to a split square-zero extensions of EFC-DGAs as encountered in Definition \ref{oplusdef}. By adjunction, this gives us a natural map 
\[
\oL\Omega^1_{\oR\sO(X)}\ten^{\oL}_{\oR\sO(X)}\sO_X^{\alg} \to \bL_X^{\an},
\]
and we wish to show that this is an equivalence.

For affine spaces, we have that $\sO(\bA^n_K)=\oT[x_1, \ldots, x_n]$ is cofibrant, and hence $\oL\Omega^1_{\oR\sO(\bA^n_K)}\simeq \Omega^1_{\sO(\bA^n_K)}\cong \sO(\bA^n_K)^n$, giving the required isomorphism in this case.
In general, since $\pi_0\map(X,\bA^1_K) \cong \H_0\oR\sO(X) \cong \Hom(t_0X,\bA^1_K)$ and $t_0X$ is finitely embeddable, there exists a closed immersion $X \to \bA^n_K$ for some $n$.  The general result then follows from  Lemma \ref{cotalglemma} and  \cite[Corollary 5.32]{PortaYuRep}, which show that both relative  $\bL^{\an}$ and $\oL\Omega^1_{\oR\sO}$ are given by algebraic cotangent complexes on closed immersions.  
\end{proof}

\begin{proposition}\label{DHomprop}
 If $Y$ and $X$ are  derived $K$-analytic spaces in the sense of \cite[Definition 12.3]{lurieDAG9} or \cite[Definition 2.5]{PortaYuNonArch} such that the underived truncation $t_0Y$ (resp. $t_0X$) is equivalent to the structured topos associated to an analytic space (resp.  a  finitely embeddable Stein space), then the natural map
\[
 \map_{\mathrm{dAn}_K}(Y, X) \to \map_{EFC-DGA_K}(\oR\sO(X), \oR\sO(Y))
\]
of mapping spaces (coming from contravariant functoriality of $\oR\sO$) is a weak equivalence of simplicial sets.
\end{proposition}
\begin{proof}
 We prove this by induction on the Postnikov tower of the structure sheaf $\bO_Y$ of $Y$. When $\bO_Y \simeq \pi_0\bO_Y$, we have $Y \simeq t_0Y$, so the statement follows from Lemmas \ref{Homlemma} and \ref{Homlemmacoprod}, together with full faithfulness (\cite[Theorem 12.8]{lurieDAG9} and \cite[Theorem 4.11]{PortaYuNonArch}) of the functor from analytic spaces to derived analytic spaces.

By \cite[Corollary 5.42]{PortaYuRep}, the map $t_{\le n} Y \to t_{\le n+1}Y$ on truncations is an analytic square-zero extension in the sense of \cite[Definition 5.39]{PortaYuRep}. This gives us a morphism
\[
 \eta_d\co \map_{\mathrm{dAn}_K}(t_{\le n}Y  , X) \to \map_{\mathrm{dAn}_K}((Y, \tau_{\le n}\bO_{Y} \oplus \pi_{n+1}\bO_Y^{\alg}[n+1])  , X)
\]
(notation as in  \cite[Definition 5.14]{PortaYuRep}), whose homotopy equaliser with the zero section $\eta_0$ is just $\map_{\mathrm{dAn}_K}(t_{\le n+1}Y  , X) $. 

For each homotopy class $[f] \in \pi_0\map_{\mathrm{dAn}_K}(t_{0}Y  , X)$, this gives us a homotopy fibration sequence
\[
 \map_{\mathrm{dAn}_K}(t_{\le n+1}Y, X)_{[f]} \to \map_{\mathrm{dAn}_K}(t_{\le n}Y, X)_{[f]} \to \map_{\Mod(\bO_X^{\alg})}(\bL_X, \oR f_*\pi_{n+1}\bO_Y^{\alg}[n+1]). 
\]
 Proposition \ref{compcotprop} allows us to rewrite the final space as
\begin{align*}
 &\map_{\Mod(\oR(\sO(X)))}(\oL\Omega^1_{\oR\sO(X)}, \oR\Gamma(X,\oR f_*\pi_{n+1}\bO_Y^{\alg}[n+1]))\\
& \simeq  \map_{\Mod(\oR(\sO(X)))}(\oL\Omega^1_{\oR\sO(X)}, \oR\Gamma(Y,\pi_{n+1}\bO_Y^{\alg})[n+1]), 
\end{align*}

where the $\oR\sO(X)$-module structure on the right comes from the morphism $f^* \co \oR\sO(X) \to \H_0\oR\sO(Y)$. 

Now, applying the functor $\oR\sO(-)$ to the analytic square-zero extension of \cite[Corollary 5.42]{PortaYuRep}, we realise the EFC-DGA $\oR\sO(t_{\le n+1}Y)$ as 
the homotopy fibre of the EFC derivation
\[
 \oR\sO(t_{\le n}Y) \to   \tau_{\ge 0}(\oR\Gamma(Y,\pi_{n+1}\bO_Y^{\alg})[n+1])
\]
induced by $\eta_d$. For $[f] \in \pi_0\map_{\mathrm{dAn}_K}(t_{0}Y  , X)$, we thus realise have  $\map_{EFC-DGA_K}(\oR\sO(X),\oR\sO(t_{\le n+1}Y))_{[f]}$ as  
the homotopy fibre of 
\[
\map_{\mathrm{dAn}_K}(\oR\sO(X),\oR\sO(t_{\le n}Y))_{[f]} \to \map_{\Map(\oR\sO(X))}(\oL\Omega^1_{\oR\sO(X)},\oR\Gamma(Y,\pi_{n+1}\bO_Y^{\alg})[n+1] ). 
\]
Comparing the respective sequences now gives the required inductive step, giving equivalences for each $t_{\le n}Y$. The result then follows because $Y\simeq \ho\LLim_n t_{\le n}Y$, with $\oR\sO(Y)\simeq \ho\Lim_n \oR\sO(t_{\le n}Y)$.
\end{proof}

\begin{remark}
In particular, Proposition \ref{DHomprop} exhibits the $\infty$-category of Lurie's derived Stein spaces as a full subcategory of the $\infty$-category of EFC-DGAs, for both complex and non-Archimedean contexts. The key to the comparison is in comparing the respective cotangent complexes associated to closed immersions, and this will have analogues in any Fermat theory, because \cite[Remark 11.13]{lurieDAG9} amounts to saying  that the results of \cite[\S 11]{lurieDAG9} hold in that generality. 
\end{remark}

\section{DG analytic spaces and stacks}

Our results so far show that finitely presented EFC-DGAs give a natural derived extension of the category of the category of globally finitely presented Stein spaces. The idea now is just to glue these together to give DG analytic spaces and stacks. In the complex setting, this gives derived enhancements of any analytic space, because $\Cx$-analytic spaces admit an open cover by Stein spaces (open polydiscs and their closed subspaces being Stein), corresponding to localisations of EFC-algebras. 

In the non-Archimedean setting, we have to be more careful because admissible opens in the rigid analytic context do not correspond to  localisations (in the sense of Definition \ref{finlocKdef})  of EFC-algebras on their rings of convergent functions, since they  cannot necessarily be factorised in terms of open immersions of Stein spaces.   By \cite[Theorem 2.27 and proof of Theorem 2.26]{GrosseKloenne}, any partially proper rigid analytic space has a basis of open Stein subspaces, but we can do better via dagger spaces. Admissible open immersions do give rise to localisations on rings of overconvergent functions, essentially because open polydiscs are Stein.


\subsection{Analytic spaces, Deligne--Mumford and Artin analytic stacks}\label{stacksn}

By taking the opposite model category $DG^+\Aff_{\oT} $ (or a pseudo-model subcategory)  to non-negatively graded $\oT$-DGAs as our analogue of affine schemes, we can now develop the theory of derived stacks (by analogy with \cite{hag2,lurie}) with respect to any precanonical Grothendieck topology, and a suitable class $\oC$ of morphisms, for any rational Fermat theory $\oT$. In the algebraic setting, the most common classes of morphisms to consider are \'etale surjections (giving rise to Deligne--Mumford stacks) and smooth surjections (giving rise to Artin stacks). The general properties such a class of morphisms must satisfy are summarised in \cite[Properties \ref{stacks2-Cprops}]{stacks2}  or \cite[Definition 2.10]{PortaYuGAGA}, giving rise to the theory of $n$-geometric stacks as certain simplicial presheaves on non-negatively graded $\oT$-DGAs (beware of discrepancies in the value of $n$ between references, stemming from different versions of \cite{hag2}). 

By analogy with the strongly quasi-compact $n$-geometric stacks of \cite{hag2}, the most natural objects to develop are built from objects of $DG^+\Aff_{\oT} $  rather than arbitrary disjoint unions. Since countable disjoint unions of Stein spaces are Stein, this is not much of a restriction in the cases we are interested in;
  quasi-compactness in these analytic settings is relative to a Grothendieck topology generated by finite   covering families of globally finitely presented Stein spaces.  
Such  strongly quasi-compact $n$-geometric stacks are characterised in \cite[Theorem \ref{stacks2-relstrict}]{stacks2}: they admit representations by $(n,\oC)$-hypergroupoids, i.e. simplicial diagrams in $DG^+\Aff_{\oT} $ satisfying a form of Kan condition, namely that the partial matching maps are all $\oC$-morphisms, and are eventually equivalences. 

For a class $\oE$ of morphisms as in Assumption \ref{openetale}, we can take $\oC_{\oE,\mathrm{DM}}$ to be the class generated under pushout, composition and retraction by weak equivalences and faithfully flat $\oE$-morphisms. This consists of morphisms $A \to B$ for which $\H_0A \to \H_0B$ is in $\oE$ and $\H_*B\cong \H_*A\ten_{\H_0A}\H_0B$, and gives rise to analogues of derived Deligne--Mumford stacks. We can also consider an analogue of derived Artin stacks by taking $\oC_{\oE,\mathrm{Artin}}$ to be the class generated by  weak equivalences, faithfully flat $\oE$-morphisms, and the map $\Q_{\oT} \to \Q[x]_{\oT}$, provided all pushouts of  the latter are flat. The argument of \cite[4.3.22]{lurieDAG5} ensures that the resulting Grothendieck topologies generated by finite covering families are precanonical.


\begin{examples}
When we take  $\oT=\Com_R$ and $\oE$ the class of \'etale morphisms, then for the class $\oC_{\oE,\mathrm{DM}}$ as above, the theory of $ (n,\oC_{\oE,\mathrm{DM}})$-hypergroupoids in $DG^+\Aff_{\Com_R} $ is just the theory of strongly quasi-compact $n$-geometric derived  Deligne--Mumford  stacks over $R$. For the class $\oC_{\oE,\mathrm{Artin}}$, the theory of $ (n,\oC_{\oE,\mathrm{Artin}})$-hypergroupoids is just the theory of strongly quasi-compact $n$-geometric derived Artin  stacks.

When $\oT$ is the theory of EFC $\Cx$-algebras  and we take $\oE=\et$, consisting of  \'etale maps between globally finitely presented Stein spaces, then for the classes $\oC_{\et,\mathrm{DM}}$ (resp. $\oC_{\et,\mathrm{Artin}}$)  above, 
\cite[Lemma 3.3]{PortaGAGA1},
 \cite[Theorem \ref{stacks2-bigthm}]{stacks2} and Proposition \ref{DHomprop} imply that the $\infty$-category of $ (1,\oC_{\et,\mathrm{DM}})$-hypergroupoids in $DG^+\Aff_{\oT} $ contains as a full subcategory those 
derived $\Cx$-analytic spaces $X$ of \cite[Definition 12.3]{lurieDAG9} for which $t_0X$ is a $\Cx$-analytic space with global bounds on the numbers of generators and relations of its local rings.
 By Remark \ref{embeddingrmk}, the latter  condition  ensures that $t_0X$ has a finite  atlas by globally finitely presented Stein spaces
; to drop the condition, we would have to replace $DG^+\Aff_{\oT} $ with \'etale sheaves on  $DG^+\Aff_{\oT} $ given by coproducts of representables. 

Likewise, the $\infty$-category of $ (n,\oC_{\et,\mathrm{DM}})$-hypergroupoids in $DG^+\Aff_{\oT} $
(resp. $ (n,\oC_{\et,\mathrm{Artin}})$-hypergroupoids) in $DG^+\Aff_{\oT} $ contains as a full subcategory those Deligne--Mumford (resp. Artin) $(n-1)$-geometric derived $\Cx$-analytic stacks $\fX$ of 
  \cite[Definition 8.2]{PortaGAGA1} which satisfy global finite presentation conditions on $t_0\fX$ (which could be dropped if we were allow arbitrary coproducts of representables in our definition).

When $\oT$ is the theory of EFC $K$-algebras over a non-Archimedean field $K$, and we take $\oE=\et$, consisting of  \'etale morphisms, \cite[Theorem \ref{stacks2-bigthm}]{stacks2} and Proposition \ref{DHomprop}  imply that the $\infty$-category of $ (n,\oC_{\pro(\et),\mathrm{DM}})$-hypergroupoids (resp. $ (n,\oC_{\pro(\et),\mathrm{Artin}})$-hypergroupoids) in $DG^+\Aff_{\oT} $,  contains as a full subcategory a theory of Deligne--Mumford (resp. Artin) $n$-geometric $K$-dagger stacks $\fX$ analogous to the rigid derived analytic stacks of \cite{lurieDAG9}, equipped with finite presentation conditions on $t_0\fX$. The differences arise because we need to be able to cover objects with Stein (or pro-Stein) spaces, so either have to work with overconvergent functions or restrict to objects satisfying partial properness conditions, including on the higher diagonals. 

If we  take $\oT$ to be $\C^{\infty}$ and $\oE$ to be the class of local diffeomorphisms, then  $ (n,\oC_{\oE,\mathrm{DM}})$-hypergroupoids (resp. $ (n,\oC_{\oE,\mathrm{Artin}})$-hypergroupoids) in $DG^+\Aff_{\oT} $ give settings for derived differential geometry, as in \cite{nuitenThesis} 
\end{examples}

In addition to the general results of \cite[\S \ref{stacks2-ressn}]{stacks2} which hold for any geometric context, the results of \cite[\S\S \ref{stacks2-dsheaves} and \S \ref{stacks2-defsn}]{stacks2} adapt verbatim from derived algebraic geometry to models of derived geometry based on the classes $\oC_{\oE,\mathrm{DM}}$ and  $\oC_{\oE,\mathrm{Artin}}$ above, on the model category $DG^+\Aff_{\oT}$ for a rational Fermat theory $\oT$. More precisely, we have a theory of quasi-coherent sheaves arising as Cartesian complexes on our simplicial diagrams in $DG^+\Aff_{\oT}$. This leads to a theory of cotangent complexes on replacing algebraic K\"ahler differentials with Definition \ref{Omegadef}, and these govern deformations.

\begin{definition}
 Given $X \in DG^+\Aff_{\oT}$ corresponding to a $\oT$-DGA $\sO(X)$, we write $\pi^0X \in DG^+\Aff_{\oT}$ for the object corresponding to the $\oT$-algebra $\H_0\sO(X)$.
\end{definition}

\begin{proposition}\label{dgshfthm}
For a $\Cx$-analytic space (resp. a non-Archimedean $K$-dagger space) $Z$ and $m \ge 0$, the  $\infty$-category of     ($m$, \'etale)-geometric  (resp. ($m$, pro-\'etale)-geometric) derived stacks over $DG^+\Aff_{EFC,\Cx}$ (resp. $DG^+\Aff_{EFC,K}$)
 with underived truncation  $\pi^0X \simeq Z$ is equivalent to the $\infty$-category  of  presheaves
 $\sA_{\bt}$ of non-negatively graded EFC-DGAs on  the category of open Stein subspaces (resp. open dagger affinoid subspaces) of $Z$,  satisfying the following:
\begin{enumerate}
\item $\H_0(\sA_{\bt})= \O_{Z}$;
\item for all $i$, the presheaf $\H_i(\sA_{\bt})$ is a quasi-coherent $\O_{Z}$-module.
\end{enumerate}
In particular, the corresponding homotopy categories are equivalent.
\end{proposition}
\begin{proof}
This is a consequence of Lemmas \ref{Homlemma}, \ref{HomlemmaNonarch}, Examples \ref{openetaleexamples}  and Proposition \ref{dgshrink}, with the same proof as \cite[Theorems \ref{stacks2-lshfthm}, \ref{stacks2-hshfthm} and \ref{stacks2-dgshfthm}]{stacks2}, and the observation that open immersions of dagger affinoids arise as inverse limits of open immersions of Stein spaces. Explicitly, for an $(m+1)$-geometric hypergroupoid $\tilde{X}$ in $DG^+\Aff_{EFC}$ with a map  $a\co  \tilde{Z}:=\pi^0\tilde{X}\to Z$ realising $\tilde{Z}$ as a simplicial hypercover of $Z$ by  Stein spaces, we set 
\[
 \sA_{\bt}:= \oR a_* (\sO(\tilde{X})/\O_{\tilde{Z}})^{l,\et}.\qedhere
\]
\end{proof}

\subsection{Comparison with structured topoi}

\begin{definition}
 Given a pair $X=(\pi^0X, \sA_{\bt})$ as in  Proposition \ref{dgshfthm} and an open Stein submanifold $U \subset \Cx^n$, define a simplicial hypersheaf $\bO_X(U)$ on the site of Stein open submanifolds of $\pi^0X$ by $V \mapsto \map_{EFC-DGA}(\sO(U), \sA_{\bt}(V))$. Write $X_{\mathrm{str}}$ for the pair $(\pi^0X,\bO_X)$. 
\end{definition}

\begin{proposition}\label{DHomprop2}
In the complex analytic setting, given a pair $X=(\pi^0X, \sA_{\bt})$ as in  Proposition \ref{dgshfthm} for which the quasi-coherent sheaves $\H_i\sA_{\bt}$ are coherent on $\pi^0X$, 
the data $X_{\mathrm{str}}=(\pi^0X,\bO_X)$ determine a derived $\Cx$-analytic space in the sense of \cite[Definition 12.3]{lurieDAG9}. This construction forms a left inverse to the functor $Y \mapsto (\fY, \oR\sO_Y)$ of Definition \ref{ROdef} when restricted to  derived $\Cx$-analytic  spaces $Y$ for which $t_0Y$ is an analytic space.
\end{proposition}
\begin{proof}
 Open Stein submanifolds of affine space form a pregeometry  $\cT^0_{\mathrm{Stein},\Cx}$ in the sense of  \cite[Definitions 1.2.1 and 3.1.1]{lurieDAG5}, with admissible morphisms the local biholomorphisms, and  
Grothendieck topology generated by covering families of open immersions. We now check that  $\bO_X$ defines a  $\cT^0_{\mathrm{Stein},\Cx}$-structure in the sense of \cite[Definition 3.1.4]{lurieDAG5}:

\begin{enumerate}
 \item We first need to show that $\bO_X$ preserves finite products, which  amounts to saying that for $U,V \in \cT^0_{\mathrm{Stein},\Cx}$, we have a quasi-isomorphism $\sO(U\by V) \simeq \sO(U)\odot^{\oL}\sO(V)$ of EFC-DGAs. Since $\sO(U\by V)\simeq \sO(U)\odot\sO(V)$, this reduces to the observation that Proposition  \ref{locmodelprop} allows us to calculate homotopy pushouts using the \'etale model structure on EFC-DGAs, in which $\sO(U)$ and $\sO(V)$ are cofibrant.

\item We now need to show that $\bO_X$ preserves pullbacks along admissible morphisms (i.e. open immersions). For 
morphisms $U \to W \la V$  in $\cT^0_{\mathrm{Stein},\Cx}$ with $U \to W$ an open immersion, this amounts to showing that $\sO(V\by_XU) \simeq \sO(U)\odot^{\oL}_{\sO(X)}\sO(V)$. Since all these objects are cofibrant in the \'etale
model structure, and one morphism is a cofibration, this follows from \cite[Proposition 13.1.2]{Hirschhorn}. 

\item If $\{U_{\alpha} \to W\}_{\alpha}$ is an open covering in $\cT^0_{\mathrm{Stein},\Cx}$, then we need to show that
\[
 \coprod_{\alpha} \bO_X(U_{\alpha}) \to \bO_X(W) 
\]
is an effective epimorphism of simplicial hypersheaves on $\pi^0X$. This amounts to saying that the map
\[
 \coprod_{\alpha} \pi_0\map_{EFC-DGA}(\sO(U_{\alpha}), \sO_X) \to\pi_0\map_{EFC-DGA}(\sO(W), \sO_X)  
\]
of set-valued presheaves induces a surjection on sheafification.
By Lemma \ref{mapSteinlemma}, we have $\pi_0\map(\sO(U),\sO_X(V)) \cong \Hom_{EFC}(\sO(U),\sO_{\pi^0X}(V))$ for any $U \in \cT^0_{\mathrm{Stein},\Cx}$. This condition then follows from \cite[Lemma 12.6]{lurieDAG9}, which shows that $\bO_{\pi^0X}$ is a $\cT_{\an}$-structure.
\end{enumerate}

It follows from \cite[Remark 11.3]{lurieDAG9} and \cite[Proposition 3.2.8 and Remark 4.4.2]{lurieDAG5} that $\cT^0_{\mathrm{Stein},\Cx}$ is Morita  equivalent in the sense of \cite[Definition 3.2.2]{lurieDAG5}  to the  pregeometry $\cT_{\mathrm{Stein}}$ of \cite[Definition 4.4.1]{lurieDAG5} or the pregeometry $\cT_{\an}$ of \cite[Definition 11.1]{lurieDAG9}. 

Thus $\bO_X$ determines a $\cT_{\an}$-structure on the analytic site of $\pi^0X$. Since the Dold--Kan normalisation of $\bO_X^{\alg}$ is equivalent to $\sA_{\bt}$, we have  $\pi_i\bO_X^{\alg} \cong \H_i\sA_{\bt}$, which is coherent on $\pi^0X$, meaning that $(\pi^0X, \bO_X)$ is a derived analytic space.

Finally, if $Y$ is a derived analytic space, then the hypersheaf $\oR\sO_Y$ of EFC-DGAs gives an EFC structure on the Dold--Kan normalisation of $\bO_Y^{\alg}$. In particular, each $\H_i\oR\sO_Y$   is coherent on $t_0Y$, so $(t_0Y, \oR\sO_Y)$ satisfies the conditions of Proposition \ref{dgshfthm} whenever $t_0Y$ is an analytic space. For each  $U \in \cT^0_{\mathrm{Stein},\Cx}$, functoriality gives a map $\bO_Y(U)\to \map_{EFC-DGA}(\sO(U), \oR\sO_Y)$, and hence a natural transformation  $ (t_0Y, \oR\sO_Y)_{\mathrm{str}} \to Y$. That this is an equivalence follows because when $X= (t_0Y, \oR\sO_Y)$, we have $\pi_i\bO_X^{\alg} \cong \H_i\oR\sO_Y\cong \pi_i\bO_Y^{\alg}$. 
\end{proof}

\begin{remark}\label{DHomrmk}
Although Proposition \ref{DHomprop2} demonstrates that the $\cT_{\an}$-structure on a derived complex affine scheme is determined by its restriction to the   subcategory of affine spaces, it does not follow that the cocompact objects of  $DG^+\Aff_{EFC,\Cx}$ form a geometric envelope for $\cT_{\an}$  in the sense of \cite[Definition 3.4.1]{lurieDAG5}. Proposition \ref{locmodelprop} ensures that there is a natural map from the geometric envelope of  $\cT^0_{\mathrm{Stein},\Cx}$ to  $DG^+\Aff_{EFC,\Cx}$. Part of the proof of \cite[Proposition 4.2.3]{lurieDAG5} adapts to give a full and faithful embedding of $DG^+\Aff_{EFC,\Cx}$ in the geometric envelope by Kan extension, 
 and the coherence conditions on finitely embeddable derived Stein spaces just happen to ensure that they lie in its image.

In the non-Archimedean setting of \cite{PortaYuNonArch}, the inverse functor to Proposition \ref{DHomprop} is harder to describe, because affinoids are not Stein. However, if we modify \cite{PortaYuNonArch} to replace $\cT_{\an}(K)$ with smooth dagger analytic spaces (so take overconvergent functions)   or with partially proper smooth analytic spaces $U$, then  $V \mapsto \map_{EFC-DGA}(\sO(U), \sA_{\bt}(V))$ again defines a $\cT$-structured topos.
\end{remark}

\subsection{Analytification}

For any rational Fermat theory $\oT$, the forgetful functor from $\oT$-DGAs to rational CDGAs has a left adjoint. For our EFC-DGAs over $K$ (where  $K$ is $\Cx$ or a non-Archimedean field), there is a forgetful functor to $K$-CDGAs, whose left adjoint we regard as analytification, denoted $A\mapsto A^{\an}$. When $X$ is an affine $K$-scheme of finite type, the EFC-algebra $\Gamma(X,\sO_X)^{\an}$ is then the ring of global functions of the analytic space associated to $X$. Since analytification preserves \'etale morphisms, it gives a functor from algebraic Deligne--Mumford (resp. Artin) $n$-stacks to $K$-analytic Deligne--Mumford (resp. Artin) $n$-stacks, via the constructions of \S \ref{stacksn}. Explicitly, if $X_{\bt}$ is a DM (resp. Artin)  $n$-hypergroupoid in affine DG schemes, then $X_{\bt}^{\an}$ is a DM (resp. Artin)  $n$-hypergroupoid in $DG^+\Aff_{\oT} $, and this construction preserves hypercovers, giving a functor on $n$-stacks. The details follow exactly as in \cite[Appendix A.4]{
DiNataleHolsteinGlobal}. 

\subsection{Symplectic and Poisson structures, and quantisations}

A useful feature of differential graded models for derived analytic geometry  is that they make the formulation of shifted Poisson structures, and also of deformation quantisations, fairly straightforward, in terms of multiderivations or differential operators on $\oT$-DGAs. For the theory of $\C^{\infty}$-DGAs, this is all laid out in \cite{DQDG}, but the same constructions extend to any rational Fermat theory with a class $\oE$ of morphisms as in Assumption \ref{openetale}. The most important technical result enabling this to work is Lemma \ref{Cunramlemma}, which provides \'etale functoriality of Poisson structures and quantisations.   Note that \cite[Remarks \ref{DQDG-htpyfdrmkd1} and \ref{DQDG-htpyfdrmkd2}]{DQDG} give the correct generality in which to operate, and the results of \cite{poisson} on shifted Poisson structures and of \cite{DQvanish,DQnonneg,DQLag,DQpoisson} on shifted quantisations all then extend from derived algebraic  to derived analytic settings. 
The constructions are compatible with analytification functors, essentially because any algebraic differential operator on a ring gives rise to an analytic differential operator on its analytification. 



\bibliographystyle{alphanum}
\bibliography{references.bib}

\end{document}